\documentclass[11pt,twoside]{article}
\usepackage{fancyhdr,color,graphicx,amsmath,amsfonts,amssymb,latexsym,bm,indentfirst,cite,amsthm,enumerate}
\usepackage[colorlinks=true,backref=page]{hyperref}
\usepackage[all]{xy}
\usepackage[a4paper,left=25mm,right=25mm,top=35mm,body={155mm,230mm}]{geometry}
\usepackage{pdfpages}
\usepackage{titletoc}
\usepackage{fancyhdr}
\newtheorem*{claim}{\hspace{2em} Claim}

\newtheorem{theorem}{\hspace{2em}Theorem}[subsection]
\newtheorem{definition}[theorem]{\hspace{2em} Definition}
\newtheorem{lemma}[theorem]{\hspace{2em}Lemma}

\newtheorem{proposition}[theorem]{\hspace{2em} Proposition}

\newtheorem{Theorem}{\hspace{2em}Theorem}[section]
\newtheorem{Definition}[Theorem]{\hspace{2em} Definition}
\newtheorem{Lemma}[Theorem]{\hspace{2em} Lemma}

\newtheorem{Proposition}[Theorem]{\hspace{2em} Proposition}

\newtheorem{theoremalph}{Theorem}

\newtheorem*{theoremm}{Main Theorem}

\makeatletter
\newcommand{\subsectionruninhead}{\@startsection{subsection}{2}{0mm}{-\baselineskip}{-0mm}{\bf\large}}
\newcommand{\subsubsectionruninhead}{\@startsection{subsubsection}{3}{0mm}{-\baselineskip}{-0mm}{\bf\normalsize}}
\makeatother

\begin{document}
  \newpage
  \title{On the F-expanding of Homoclinic class}
  \author{Wanlou Wu \and Bo Li }
  \maketitle
\begin{abstract}
  We establish a closing property for thin trapped homoclinic classes. Taking advantage of this property, we proved that if the homoclinic class $H(p)$ admits a dominated splitting $T_{H(p)}M=E\oplus_{<}F$, where $E$ is thin trapped (see Definition \ref{Def:TP}) and all periodic points homoclinically related to $p$ are uniformly $F$-expanding at the period (see Definition \ref{Def:expanding}), then $F$ is expanded (see Definition \ref{Def:TP}).
\end{abstract}
	
\section{Introduction}
  Let $f$ be a diffeomorphism on a compact manifold $M$ with metric $d$. Given two hyperbolic periodic points $p$ and $q$ of $f$, denoted by $W^s(p)$, $W^s(q)$ the stable manifolds of $p$ and $q$, denoted by $W^u(p)$, $W^u(q)$ the unstable manifolds of $p$ and $q$. We call that they are homoclinically related if $W^s(orb(p))\pitchfork W^u(orb(q))\neq\emptyset$ and $W^s(orb(q))\pitchfork W^u(orb(p))\neq\emptyset$. And we denote this relation by $p\sim q$. The homoclinic class of a hyperbolic periodic point $p$ is defined as $H(p)\triangleq\overline{\{q: q\in P(f), q\sim p\}}$, where $P(f)$ denotes the set of all hyperbolic periodic points of $f$. In general, the hyperbolicity of the periodic points contained in a compact invariant set $\Lambda$ is not enough to get the hyperbolicity of $\Lambda$. For example, Kupka-Smale Theorem \cite[$P91$]{PJ} affirms that every periodic orbit of $C^r$-generic ($r\geq 1$) diffeomorphisms is hyperbolic and the stable and unstable manifolds of the periodic orbits are pairwise transverse. It means that the homoclinic class is usually not hyperbolic although it contains many hyperbolic invariant sets: any set of periodic orbits homoclinically related to $p$. Bonatti, Gan and Yang \cite[Main Theorem]{BGY} gave a sufficient criterion for the hyperbolicity of a homoclinic class. They also raised a question: Can we obtain the hyperbolicity of an invariant compact set by using the hyperbolicity of the periodic orbits in the set?   
  
  Our main work shows that one can get some \textquotedblleft hyperbolicity\textquotedblright of $H(p)$ under the topological conditions (see Definition \ref{Def:TP}) on the periodic orbits homoclinically related to $p$. Now, we will introduce precisely the notions for these conditions.
\begin{Definition}\label{DS}
	The invariant splitting $T_{\Lambda}M=E\oplus F$ over a compact $f$-invariant set $\Lambda$ is a dominated splitting if there are two constants $C>0,~\lambda\in (0,1)$ such that $$\parallel Df^n\mid_{E_x}\parallel\cdot\parallel Df^{-n}\mid_{F_{f^n(x)}}\parallel\leq C\lambda^n,~\text{for any $x\in \Lambda$ and every $n\in \mathbb{N}$}.$$  
\end{Definition}
  
\begin{Definition}\label{Def:hyperbolic}
  A periodic point $p$ is hyperbolic if there are constants $C>0$, $\lambda\in (0,1)$ and an invariant splitting $T_pM=E_p\oplus F_p$, such that for every $n\in \mathbb{N}$, one has that $$\parallel Df^n\mid_{E_p}\parallel\leq C\lambda^n,~~\parallel Df^{-n}\mid_{F_p}\parallel\leq C\lambda^n.$$
\end{Definition}       

  For the dominated splitting $T_{\Lambda}M=E\oplus F$, a plaque family tangent to the bundle $E$ is a family of continuous maps $\mathcal{W}$ from the linear bundle $E$ to $M$ satisfying: 
\begin{flushleft}
  $~~~~\langle 1\rangle$ for each $x\in \Lambda$, the map $\mathcal{W}_x:E_x\rightarrow M$ is a $C^1$-embedding that satisfies $\mathcal{W}_x(0)=x$ and whose image is tangent to $E_x$ at $x$;\\
  $~~~~\langle 2\rangle$ $(\mathcal{W}_x)_{x\in\Lambda}$ is a continuous family of $C^1$-embeddings. 
\end{flushleft}
  Let $\mathcal{W}(x)$ be the image of embedding $\mathcal{W}_x$. A  plaque family $\mathcal{W}$ is locally invariant if there is $\delta>0$ such that for every $x\in\Lambda$, one has that $f\circ\mathcal{W}_x(B(0,\delta))\subseteq\mathcal{W}(fx)$, where $B(0,\delta)\subseteq E_x$ is the ball whose radius is $\delta$. Plaque Family Theorem \cite[Theorem 5.5]{HPS} shows that there always exists a locally invariant plaque family tangent to $E$. A plaque family is called \emph{trapped} if for every $x\in\Lambda$, one has that $f(\overline{\mathcal{W}(x)})\subseteq \mathcal{W}(fx)$.
  
\begin{Definition}\label{Def:TP}
  For the dominated splitting $T_{\Lambda}M=E\oplus F$, $E$ is thin trapped if for any neighborhood $U$ of the section $0$ in $E$, there is 
\begin{flushleft}
  $\langle 1\rangle$ a continuous family $\{\varphi_x\}_{x\in\Lambda}$ of $C^1$-diffeomorphisms of the spaces $\{E_x\}_{x\in\Lambda}$ supported in $U$;\\
  $\langle 2\rangle$ a constant $\delta>0$ such that $f(\overline{\mathcal{W}_x\circ\varphi_x(B(0,\delta))})\subseteq\mathcal{W}_{fx}\circ\varphi_{fx}(B(0,\delta))$ for any $x\in\Lambda$. 
\end{flushleft}
  $F$ is expanded on $\Lambda$, if there are $C>0$ and $\lambda\in(0,1)$ such that $$\|Df^{-n}|_{F_x}\|\leq C\lambda^n,~\text{ for any $x\in\Lambda$ and $n\in\mathbb{N}^+$}$$ 
\end{Definition}

\begin{Definition}\label{Def:expanding}
  Let $\Lambda=\{q_n\}_{n\in\mathbb{N}}$ be a sequence of hyperbolic periodic points of diffeomorphism $f$, for the dominated splitting $T_\Lambda M=E\oplus_{<}F$, $f$ is uniformly $F$-expanding at the period on $\Lambda$ if there are two constants $C>0,~\lambda\in (0,1)$ such that for any $n\in\mathbb{N}$, one has that $$\prod^{\pi(q_n)}_{j=1}\parallel Df^{-1}\mid_{F_{f^j(q_n)}}\parallel\leq C\lambda^{\pi(q_n)},$$ where $\pi(q_n)$ is the period of the periodic point $q_n$. We also call that $F$ is uniformly $\lambda$-expanding at the period on $\Lambda$.    	
\end{Definition}
  Now, we introduce our main result and the sketch of our proof for the main result.
\begin{theoremm}\label{Thm:main}
  Let $f$ be a diffeomorphism on a compact Riemannian manifold, $p$ be a hyperbolic periodic point. Assume that the homoclinic class $H(p)$ admits a dominated splitting $T_{H(p)}M=E\oplus_{<}F$ and $E$ is thin trapped with dim$E=$ind $(p)$. If $f$ is uniformly $F$-expanding at the period on all periodic points homoclinically related to $p$, then $F$ is expanded on $H(p)$.  
\end{theoremm}

 This Main Theorem also gives a criterion for getting \textquotedblleft weak periodic point\textquotedblright  which means that they have a Lyapunov exponent arbitrarily close to zero in a given homoclinic class.
 
\begin{theoremalph}\label{wpp}
 	Let $p$ be a hyperbolic periodic point of a diffeomorphism $f$ on a compact Riemannian manifold. If the homoclinic class $H(p)$ satisfying:
 	\begin{enumerate}[$\bullet$]
 		\item the homoclinic class $H(p)$ admits a  partially hyperbolic splitting $T_{H(p)}M=E^s\oplus E^c\oplus E^u$ with $E^s$ is uniformly contracted, $E^u$ is uniformly expanding and dim$E^c=1$;
 		\item dim$E^s=$ind $(p)$ and $H(p)$ is not hyperbolic. 
 	\end{enumerate}
 	then for every $\varepsilon>0$, one can find a periodic point $q\in H(p)$ homoclinically related to $p$ such that $$\frac{1}{\pi(q)}\log(\lVert Df^{\pi(q)}|_{E^c_q}\rVert)\leq\varepsilon.$$  
\end{theoremalph}

  The Main Theorem gives a criterion for getting some hyperbolicity of a homoclinic class from the same property of the periodic points in this homoclinic class. We should consider the question: how to establish the relations between non-periodic points in the compact set and periodic points? The Anosov Closing Lemma \cite[P269, Theorem 6.4.15]{KB} implies that for any point in a hyperbolic set whose orbit nearly returns to the point, there is a periodic orbit closely shadowing this nearly-returning orbit. Gan \cite[Theorem 1.1]{GS} showed that any quasi-hyperbolic pseudoorbit with recurrence can be shadowed by periodic orbit. But in our assumptions, the homoclinic class is not hyperbolic set and we can not get the quasi-hyperbolic pseudoorbit with respect to the dominated splitting on the homoclinic class. Even though the recurrent orbit (non-periodic) can be shadowed by periodic orbit, we also need to consider that: how to expand the property of periodic orbit to other orbit (non-periodic orbit).

\section{Preliminaries}  
  In this section, we introduce some important tools and some basic facts but useful for the proof of the Main Theorem. 

\subsection{Hyperbolic time}  

 Let $f$ be a diffeomorphism on a compact manifold $M$. For every $x\in M$ and $n\in\mathbb{N}^+$, $(x,n)$ denotes the segment of orbit $(x,n)\triangleq\{x,f(x),\cdots,f^{n-1}(x)\}$. We consider a compact invariant set $\Lambda$ which has a dominated splitting $T_\Lambda M=E\oplus_< F$. By \cite[P289, Appendix B]{BDV}, one can fix an admissible compact neighborhood of $\Lambda$ and denotes by $M(f,U)=\bigcap_{i\in\mathbb{Z}}f^iU$ the maximal invariant set in $U$. The dominated spliting $E\oplus_< F$ extends in a unique way on $M(f,U)$.
 
 \begin{definition} 
 	For any $\lambda\in(0,1)$ and $x\in M(f,U)$, an orbit segment $(x,n)$ is an uniform $\lambda$-string, if $\prod\limits_{j=k+1}^{n}||Df^{-1}|_{F_{f^j(x)}}||\leq \lambda^{n-k}$, for $k=0,1,\cdots,n-1$. And $n\in\mathbb{N^+}$ is called $\lambda$-hyperbolic time of $x$. 
 \end{definition}
 Denoted by $HT(x,\lambda)$ the set of all $\lambda$-hyperbolic times of $x$ and the $n$-th $\lambda$-hyperbolic time of $x$ is denoted by $HT_n(x,\lambda)$. For a periodic point $x$, denoted by $\Gamma_1(x,\lambda)$ the largest $\lambda$-hyperbolic time in the period $\pi(x)$ and the smallest $\lambda$-hyperbolic time larger than the period $\pi(x)$ is denoted by $\Gamma_2(x,\lambda)$. The following lemma, given by Pliss (\cite{PL}), gives us a tool to prove that a point satisfying the assumption that $\liminf\limits_{m\rightarrow\infty}\frac{1}{m}\sum\limits_{i=1}^{m}\log\|Df^{-1}|_{F_{f^i(x)}}\|\leq \log\lambda$, for same $\lambda\in(0,1)$, has many (positive density at infinity) hyperbolic times. 
 
\begin{lemma}[Pliss Lemma \cite{PL}]\label{Pliss}
  Given constants $A$ and $C_2<C_1<0$ with $A\geq |C_2|$, there is $\theta\in(0,1)$ such that if any real numbers $\{a_j\}_{j=1}^N$ satisfy the conditions
\begin{enumerate}[(1)]
	\item $|a_j|\leq A $, for $j=1,2,\cdots,N$;
	\item $\sum_{j=1}^{N}a_j\leq N C_2$. 
\end{enumerate}  
 then there is an integer $l\geq \theta N$ and a sequence numbers $1\leq n_1<n_2<\cdots<n_l\leq N$ such that $$\sum_{j=n+1}^{n_i}a_j\leq (n_i-n) C_1,~~~\forall~0\leq n<n_i,~i=1,2,\cdots,l.$$
\end{lemma}

\begin{lemma}\label{many hyperbolic time}
  Given $\mu<\mu_{1}<\mu_{2}<0$ and $A>|\mu_{1}|$, for any real number sequence $\{a_i\}_{i=1}^{\infty}$ with $|a_i|\leq A$, if there is a number $m\in\mathbb{N}^+$ such that $\sum_{i=1}^{m}a_i\leq m\mu$ and $a_{i+m}=a_i$ for every $i\in\mathbb{N}$, then there are $\theta\in(0,1)$ and $N\in\mathbb{N}$ such that for any $k\geq N$, there exist $1\leq n_1<n_2<\cdots<n_l\leq k$ with $l\geq k\theta$ such that $$\sum_{i=n+1}^{n_j}a_i\leq (n_i-n)\mu_2,~~~\forall~0\leq n<n_j,~j=1,2,\cdots,l.$$
\end{lemma}
\begin{proof}
   Since $|a_i|\leq A$, then $r\triangleq\max\{a_1,a_1+a_2,\cdots,\sum_{i=1}^{m}a_i\}\leq mA<+\infty$. For any $k\in\mathbb{N}$, by integer division, we may assume that $k=nm+r_0$, where $0\leq r_0<m$. Hence, $\sum_{j=1}^{k}a_j\leq nm\mu+r$. Therefore, $$\liminf_{k\rightarrow\infty}\frac{1}{k}\sum_{i=1}^{k}a_i\leq \liminf_{k\rightarrow\infty}\frac{nm\mu}{k}+\liminf_{k\rightarrow\infty}\frac{r}{k}=\mu.$$ Thus, there is $N\in\mathbb{N}$ such that $\sum\limits_{i=1}^ka_i\leq n\mu_1$, for any $k\geq N$. By Pliss Lemma \ref{Pliss}, there is $1\leq n_1<n_2<\cdots<n_l\leq k$ with $l\geq k\theta$ such that $$\sum_{i=n+1}^{n_j}a_i\leq (n_j-n)~\mu_2,~\forall~0\leq n<n_j,~j=1,2,\cdots,l.$$
\end{proof}
 
 For two consecutive hyperbolic times, we can not give the estimation as the obstruction orbit segment (see Definition \ref{obstrction point}). The following content is a simple fact that can help us deal with the obstruction orbit segment.

\begin{definition}\label{obstrction point} 
  For any $\lambda\in(0,1]$ and $x\in M(f,U)$, an orbit segment $(x,n)$ is a $\lambda$-obstruction orbit segment, if $\prod\limits_{j=1}^{k}||Df^{-1}|_{F_{f^j(x)}}||\geq\lambda^{k}$, for $k=1,\cdots,n$. And the point $x$ is a $\lambda$-obstruction point, if $(x,n)$ is a $\lambda$-obstruction orbit segment for any $n\in\mathbb{N^+}$.
\end{definition}  

\begin{lemma}\label{obstrction}
  For any $r\in(0,1)$ and $\varepsilon>0$, there exists $N_0=N_0(r,\varepsilon)$ such that for some $n_0\geq N_0$, if $(x,n_0)$ is a $r$-obstruction orbit segment, then $d(x,\Lambda(r))<\varepsilon$, where $\Lambda(r)$ is the set of all $r$-obstruction point.
\end{lemma}
\begin{proof}
  Let $\Lambda_N(r)$ be the set of points such that the orbit segment $(x,N)$ is a $r$-obstruction orbit segment. Then $\Lambda(r)=\bigcap_{N>0}\Lambda_N(r)$. By the Definition \ref{obstrction point}, one has that  $\Lambda_{N}(r)\supset\Lambda_{N+1}(r)$. Therefore, this is a decreasing intersection of a compact set. For $\varepsilon>0$, taking $N_0=N_0(r,\varepsilon)$ such that $\Lambda_{N_0}(r)$ is contained in the $\varepsilon$-neighborhood of $\Lambda(r)$. Then, $d(\Lambda_{N_0}(r),\Lambda(r))<\varepsilon$. For some $n_0\geq N_0$, if $(x,n_0)$ is a $r$-obstruction orbit segment, then $x\in\Lambda_{N_0}(r)$. Hence, $d(x,\Lambda(r))<\varepsilon$.
\end{proof}    
  
\begin{lemma}\label{infty}
  Assume that the homoclinic class $H(p)$ of hyperbolic periodic point $p$ admits a dominated splitting $T_{H(p)}M=E\bigoplus_< F$ and $F$ is uniformly $\lambda$-expanding at the period on the set of periodic points homoclinically related to $p$. If there is a $r$-obstruction point $b\in H(p)$ with $r\in(\lambda,1)$, then there exists a sequence periodic points $\{q_n: n\in\mathbb{N}^+\}\subset H(p)$ homoclinically related to $p$, such that $\lim_{n\rightarrow\infty}q_n=b$ and $HT_1(q_n,\mu)\rightarrow\infty$ as $n\rightarrow\infty$, where $\lambda<\mu<r$. Moreover, $\Gamma_2(q_n,\mu)-\Gamma_1(q_n,\mu)\rightarrow\infty$ as $n\rightarrow\infty$.
\end{lemma}

\begin{proof}
  Since $b\in H(p)$, by the definition of $H(p)\triangleq\overline{\{q:q\in P(f),q\thicksim p\}}$, there is a sequence of periodic points $\{q_n:n\in\mathbb{N}^+\}\subset H(p)$ homoclinically related to $p$, such that $\lim_{n\rightarrow\infty}q_n=b$. As $F$ is uniformly $\lambda$-expanding at the period on the set of periodic points homoclinically related to $p$, there are two constants $C>0,~\lambda\in (0,1)$ such that $\prod^{\pi(q_n)}_{j=1}\parallel Df^{-1}\mid_{F_{f^j(p_n)}}\parallel\leq C\lambda^{\pi(q_n)}$, for every $n\in\mathbb{N}^+$. Let $a_i=\log\|Df^{-1}|_{F_{f^i(q_n)}}\|$, one has that $\liminf\limits_{m\rightarrow\infty}\frac{1}{m}\sum\limits_{i=1}^{m}a_i\leq \log\lambda$. By the Lemma \ref{many hyperbolic time}, one gets that $q_n$ has infinitely many $\mu$-hyperbolic times. Now, we prove that $HT_{1}(q_n,r_2)\rightarrow\infty$ as $n\rightarrow\infty$ by contradiction. Otherwise, we suppose that there exists constant $C_1>0$ such that $HT_{1}(q_n,\mu)\leq C_1$, for all $n$. Then there is a subsequence $\{q_{n_k}\}$ of $\{q_n\}$, such that $HT_{1}(q_{n_k},\mu)$ is constant, denoting by $L$. Hence for all $q_{n_k}$, one has that $\prod_{i=1}^{L}\|Df^{-1}|_{F_{f^i(q_{n_k})}}\|\leq\mu^L$. By the continuous of $Df^{-1}$, one gets that $$\prod_{i=1}^{L}\|Df^{-1}|_{F_{f^i(b)}}\|=\lim_{k\rightarrow\infty}\prod_{i=1}^{L}\|Df^{-1}|_{F_{f^i(q_{n_k})}}\|\leq\mu^L\leq r^L.$$ This means that $b$ is not a $r$-obstraction point which contradicts our condition. Since $$\Gamma_2(q_n,\mu)-\Gamma_1(q_n,\mu)\geq\Gamma_2(q_n,\mu)-\pi(q_n)\geq HT_{1}(q_n,\mu),$$ one has that $\Gamma_2(q_n,\mu)-\Gamma_1(q_n,\mu)\rightarrow\infty$ as $n\rightarrow\infty$.
\end{proof}
 
\subsection{Exponents properties in Homoclinic class}
  
  In this section, let $p$ be a hyperbolic periodic point. The homoclinic class $H(p)$ admits a dominated splitting $T_{H(p)}M=E\bigoplus_< F$ with dim$(E)=$ind$(p)$. We assume that the bundle $E$ is thin trapped and $f$ is uniformly $F$-expanding at the period on the set of periodic points homoclinically related to $p$. First, we introduce some properties of $C^1$ diffeomerphisms. 

\begin{lemma}\label{neighbor of contraction}
   Let $f$ be a $C^1$ diffeomorphism on a compact manifold $M$. For any $x\in M$, if there is $C=C(x)>0$ and $\mu_1,~\lambda_1\in(0,1)$ such that $C\mu^n_1\leq\prod_{i=0}^{n-1}\parallel Df_{f^{i}x}\parallel\leq C\lambda_1^{n},$ $ \text { for every } n\in\mathbb{N}^+,$ then for any $\mu_2<\mu_1<\lambda_1<\lambda_2$, there is $C_0=C_0(x)$ and $r=r(\mu_1,\mu_2,\lambda_1,\lambda_2)$ such that $$C_0\mu^n_2\leq\prod_{i=0}^{n-1}\parallel Df_{f^{i}y}\parallel\leq C_0\lambda_2^n,~\forall~y\in B(x,r),~\forall~n\in\mathbb{N}^+.$$
\end{lemma}

\begin{proof}
   Since $f$ is $C^1$ diffeomorphism, for $0<\mu_2<\mu_1<\lambda_1<\lambda_2<1$, there is $r_1>0$ such that $$\frac{\mu_2}{\mu_1}\leq \frac{\parallel Df_y\parallel}{\parallel Df_x\parallel}\leq \frac{\lambda_2}{\lambda_1},~{\rm\ for \ any \ } d(x,y)\leq r_1.$$ For any $y\in B(x,r_1)$, by the Mean Value Theorem, there exists $\xi\in B(x,r_1)$ such that $d(f(x),f(y))= \parallel Df_{\xi}\parallel\cdot d(x,y)$. Since $\xi\in B(x,r_1)$, one has that $$d(f(x),f(y))\leq\frac{\lambda_2}{\lambda_1}\parallel Df_x\parallel\cdot d(x,y)\leq C\frac{\lambda_2}{\lambda_1}\lambda_1 r_1= C\lambda_2r_1.$$  Taking $r=\min\{r_1,r_1/C\}$, $C_0=\max\{C,1/C\}$, we claim that $d(f^j(x),f^j(y))\leq r_1$, for every $y\in B(x,r)$ and every $j\in\mathbb{N}^+$. Next, we prove the claim. We may assume that $d(f^j(x),f^j(y))\leq r_1$ for any $y\in B(x,r)$ and every $j=1,2,\cdots,m$. Then, for any $y\in B(x,r)$, by the Mean Value Theorem, there is $\eta\in B(x,r)$ such that 
\begin{eqnarray*}
	d(f^{m+1}(x),f^{m+1}(y))&=&\parallel Df^{m+1}_\eta\parallel d(x,y) \leq \parallel Df_{f^m\eta}\parallel \parallel Df_{f^{m-1}\eta}\parallel\cdots\parallel Df_\eta\parallel d(x,y)\\
	&\leq & \frac{\lambda_2}{\lambda_1}\parallel Df_{f^mx}\parallel\frac{\lambda_2}{\lambda_1}\parallel Df_{f^{m-1}x}\parallel\cdots\frac{\lambda_2}{\lambda_1}\parallel Df_x\parallel d(x,y)\\
	&\leq & \frac{\lambda_2^{m+1}}{\lambda_1^{m+1}}\prod_{i=0}^m\parallel Df_{f^i  x}\parallel r\leq C\frac{\lambda_2^{m+1}}{\lambda_1^{m+1}}\lambda_1^{m+1}r=C\lambda_2^{m+1}r<r_1
\end{eqnarray*}
	Therefore, $f^{m+1}(B(x,r))\subset B(f^{m+1}(x),r_1)$. This proves our claim. Therefore, for $y\in B(x,r)$ and $n\in\mathbb{N}^+$, one has that   $$\frac{\mu^{n+1}_2}{\mu^{n+1}_1}\leq\frac{\prod_{i=0}^n\parallel Df_{f^iy}\parallel}{\prod_{i=0}^n\parallel Df_{f^ix}\parallel}\leq\frac{\lambda_2^{n+1}}{\lambda_1^{n+1}}.$$ Consequently, one has that $C_0\mu_2^n\leq\prod_{i=0}^{n-1}\parallel Df_{f^iy}\parallel\leq C_0\lambda_2^n$, for $y\in B(x,r)$ and $n\in\mathbb{N}^+$.
\end{proof}

\begin{proposition}\label{USM}
	Assume that the homoclinic class $H(p)$ admits a dominated splitting $T_{H(p)}M=E\oplus_<F$ with dim$E=$ind $(p)$. If $f$ is uniformly $F$-expanding at the period on all periodic points homoclinically related to $p$, then there exists constant $N\in\mathbb{N}^+$ such that a local plaque family tangent to the bundle $F$ are the unstable manifolds of all hyperbolic periodic points with period larger than $N$.
\end{proposition}
\begin{proof}
	Since $f$ is uniformly $F$-expanding, by the Definition \ref{Def:expanding}, there are two constants $C>0,~\lambda\in (0,1)$ such that $$\prod^{\pi(x)}_{j=1}\parallel Df^{-1}\mid_{F_{f^j(x)}}\parallel\leq C\lambda^{\pi(x)},~\text{for any hyperbolic periodic point $x\in H(p)$.}$$ Let $N$ be the smallest constant which satisfies that $C\lambda^N<1$, define $$A\triangleq\{x\in H(p):x \text{ is periodic point with period larger than $N$}\}.$$ For any $x\in A$ and $\varepsilon>0$, since $T_{H(p)}M=E\oplus_<F$ is a dominated splitting, by Plaque Family Theorem \cite[Theorem 5.5]{HPS}, there always exists a locally invariant plaque family tangent to the bundle $E$ and $F$. Denote by $\mathcal{W}_{\varepsilon}^{F}(x)$ the plaque family tangent to the bundle $F$ at point $x$. Since $x$ is a hyperbolic periodic point, we may assume that $T_xM=E^s\bigoplus F^u$ is the hyperbolic splitting. By the Definition \ref{Def:hyperbolic}, there are constants $C_1>0$, $\lambda_1\in (0,1)$ such that for every $n\in \mathbb{N}$, one has that 
	$$\parallel Df^n\mid_{E^s}\parallel\leq C_1\lambda_1^n,~~\parallel Df^{-n}\mid_{F^u}\parallel\leq C_1\lambda_1^n.$$ Since dim$E=$ind $(p)$, we only need to prove that $F\subseteq F^u$. If $F\nsubseteq F^u$, then there is a vector $v\in F$ such that $v\notin F^u$. Thus, one has a decomposition $v=v^s\oplus v^u$, where $0\neq v^s \in E^s$ and $v^u\in F^u$. Therefore, for every $m\in\mathbb{N}^+$, one has that $$C_1^{-1}\lambda_1^{-m\pi(x)}\lVert v^s\rVert\leq \lVert Df^{-m\pi(x)}v\rVert\leq (\prod^{m\pi(x)}_{j=1}\parallel Df^{-1}\mid_{F_{f^j(x)}}\parallel) \lVert v\rVert\leq (C\lambda^{\pi(x)})^m\lVert v\rVert.$$  For $m$ large enough, we have that $C_1\lambda_1^{-m\pi(x)}\lVert v^s\rVert>1$ and $(C\lambda^{\pi(x)})^m\lVert v\rVert<1$. This means that our assumption $F\nsubseteq F^u$ is fault. Consequently, $F\subseteq F^u$. 
\end{proof}

   Next, we introduce that in the homoclinic class $H(p)$, there is a dense set such that every point in this set has stable manifolds of uniformly size. Given $\varepsilon>0$, a sequence points $\{x_0,\cdots, x_m\}$ is called a periodic $\varepsilon$-orbit or periodic pseudoorbit, if $x_m=x_0$ and $d(f(x_i),x_{i+1})<\varepsilon$, for $i=0,\cdots,m-1$.
   
\begin{theorem} \label{AC}(\cite{KB}, Anosov Closing Lemma)
   If $\Lambda$ is a hyperbolic set of diffeomorphism $f$, then there is an open neighborhood $U$ of $\Lambda$ and two constants $C>0$, $\varepsilon_0>0$ such that for any $\varepsilon\in(0,\varepsilon_0)$ and any periodic $\varepsilon$-orbit $\{x_0,\cdots,x_m\}\subset U$, there exist a periodic point $y\in U$ satisfying $f^m(y)=y$ and $d(f^k(y),x_k)<C\varepsilon$, for $k=0,1,\cdots,m-1$.  	
\end{theorem} 

\begin{theorem}\label{SM}
	Let $p$ be a hyperbolic periodic point of diffeomorphism $f$. Assume that the homoclinic class $H(p)$ has a dominated splitting $T_{H(p)}M=E\bigoplus_< F$ with dim$(E)=ind(p)$. If $E$ is thin trapped and $f$ is uniform $F$-expanding at the period on the set of periodic points which are homoclinically related to $p$, then for any $\varepsilon>0$, there is a $\varepsilon$-dense set $\mathcal{P}\subseteq H(p)$ of hyperbolic periodic points homoclinically related to the orbit of $p$, such that every point $q\in \mathcal{P}$ has stable manifolds of uniformly size.
\end{theorem}

\begin{proof} 
	We may assume that $p$ is a hyperbolic fix point ( if not, consider $g=f^{\pi(p)}$). Since dim$(E)=ind(p)$, for the dominated splitting $T_pM=E_p\bigoplus_< F_p$, by \cite[Theorem 1]{GN07}, taking suitable Riemann norm, there exists $\lambda_1\in(0,1)$ such that $\|Df|_{E_p}\|\leq\lambda_1$ and $\|Df^{-1}|_{F_p}\|\leq\lambda_1$. Hereafter, we fixed the numbers $0<\lambda_1<\lambda_2<\lambda_3<\lambda_4<1$. Since $f$ is $C^1$, for $\lambda_2\in(\lambda_1,1)$, there is $r>0$ such that for any $x,y\in H(p)$ with $d(x,y)<r$, one has that $$\frac{\|Df|_{E_x}\|}{\|Df|_{E_y}\|}\leq\frac{\lambda_2}{\lambda_1} {\rm     ~~~~~~and~~~~~~     }\frac{\|Df^{-1}|_{F_x}\|}{\|Df^{-1}|_{F_y}\|}\leq\frac{\lambda_2}{\lambda_1}.$$ By the definition of $H(p)$, for any $\varepsilon<r$, one can take a $\frac{\varepsilon}{2}$-dense set $$B=\{x\in H(p):x\in W^s(p)\pitchfork W^u(p)\} {\rm ~in~ } H(p).$$
	
\begin{claim} 
	For every $x\in B$, $\Lambda_x\triangleq Orb(p)\bigcup Orb(x)$ is a hyperbolic set.
\end{claim}
	
\begin{proof} 
	Since $x\in W^s(p)\pitchfork W^u(p)$, there exists $n_0\in\mathbb{N}$ such that for any $n\geq n_0$, one has that $d(f^n(x),p)<r$ and $d(f^{-n}(x),p)<r$. Therefore,  $$\frac{\|Df|_{E_{f^n(x)}}\|}{\|Df|_{E_p}\|}\leq\frac{\lambda_2}{\lambda_1},~~~\frac{\|Df^{-1}|_{F_{f^{-n}(x)}}\|}{\|Df^{-1}|_{F_p}\|}\leq\frac{\lambda_2}{\lambda_1},\text{~~for any $n\geq n_0$}.$$ Let $C_1=\max\limits_{y\in H(p)}\{\frac{\|Df_y\|}{\lambda_2},\cdots,\frac{\|Df_y^{2n_0}\|}{\lambda_2^{2n_0}},\frac{\|Df^{-1}_y\|}{\lambda_2},\cdots,\frac{\|Df_y^{-2n_0}\|}{\lambda_2^{2n_0}}\}$, taking $C=\max\{1,C_1\}$, for any $y\in\Lambda_x$ and any $n\in\mathbb{N}$, one has that $$\|Df^n|_{E_y}\|\leq C\lambda_2^n,~~~\|Df^{-n}|_{F_y}\|\leq C\lambda_2^n.$$Hence, $\Lambda_x=Orb(p)\bigcup Orb(x)$ is a hyperbolic set.
\end{proof}
	
	Fixed $x\in B$, by the Theorem \ref{AC}, for the hyperbolic set $\Lambda_x$, there exist an open neighborhood $U\supset\Lambda_x$ and $C_2,~\varepsilon_0>0$, such that for any $\varepsilon_1\in(0,\varepsilon_0)$ and any periodic $\varepsilon_1$-orbit $\{x_0,\cdots,x_m\}\subset U$, there is a periodic point $y\in U$ satisfying $f^m(y)=y$ and $d(f^k(y),x_k)<C_2\varepsilon_1$, for $k=0,1,\cdots,m-1$. Taking $\varepsilon_2=\min \{\varepsilon/2C_2,~\varepsilon/2,~r\}$, since $x\in W^s(p)\pitchfork W^u(p)$, there exists $m_0\in\mathbb{N}$ such that for any $m\geq m_0$, one has that $d(f^m(x),p)<\varepsilon_2/2$ and $d(f^{-m}(x),p)<\varepsilon_2/2$. Let $K=\sup\limits_{y\in H(p)}\{\|Df_y\|,~2\}<+\infty$, for $\lambda_2<\lambda_3<1$, taking an integer $m_1\geq C/2\log(\lambda_2/\lambda_3)$, where $C=\log\lambda_3+(2m_0-2)\log\lambda_2-(2m_0-1)\log K$. One can get a $\varepsilon_2$-closed orbit segment: $$\{f^{-m_1}(x),f^{-m_1+1}(x),\cdots,f^{-m_0}(x),f^{-m_0+1}(x),\cdots,f^{-1}(x),x,f(x),\cdots,f^{m_0}(x),\cdots,f^{m_1}(x)\}.$$ By the Theorem \ref{AC}, the closed orbit segment can be $\varepsilon/2$-shadowed by a periodic point $q_x$ with $\pi(q_x)=2m_1+1$. Therefore, $$\prod_{i=0}^{\pi(q_x)-1}\|Df|_{E_{f^i(q_x)}}\| \leq K^{2m_0-1}\cdot\lambda_2^{2m_1-2m_0+2}\leq\lambda_3^{\pi(q_x)}.$$ Hence, one has that $$\prod_{i=0}^{n\pi(q_x)-1}\|Df|_{E_{f^i(q_x)}}\|\leq\lambda_3^{n\pi(q_x)},~\text{for any $n\in\mathbb{N}$ }.$$ Let $\mathcal{P}=\{q_x:x\in B\}$, then $\mathcal{P}\subseteq H(p)$ is a set of hyperbolic periodic points homoclinically related to the orbit of $p$. For $\lambda_3<\lambda_4$, any $q\in \mathcal{P}$ and any $n\in\mathbb{N}$, by the Lemma \ref{Pliss}, there are $\theta=\theta(\lambda_3,\lambda_4)\in(0,1)$ and positive integers $n_1<n_2<\cdots<n_l\leq n$ with $l\geq \theta n\pi(q)$, such that $$\prod_{i=k}^{n_j-1}\|Df|_{E_{f^i(q)}}\|\leq\lambda_4^{n_j-k},~{\rm for    }~\forall~k\in\{0,1,\cdots,n_j-1\},~j=1,2,\cdots,l.$$ Since $q$ is a periodic point, if $n\rightarrow\infty$, then there is a point $q'$ which is a iteration of $q$, such that  $$\prod_{i=0}^{m-1}\|Df|_{E_{f^i(q')}}\|\leq\lambda_4^m,\text{~for any $m\in\mathbb{N}$}.$$ It means that $q'$ has stable manifolds of $\delta$-size, where $\delta$ is only relate to $\lambda_4$. Since $E$ is thin trapped, for every $y\in\{q',\cdots,f^{\pi(q')-1}(q')\}$ and $\delta>0$, one has that $f^i(\mathcal{W}_{\delta}^{s}(y))\subset\mathcal{W}_{\delta}^{s}(f^iy)=\mathcal{W}_{\delta}^{s}(q')$, for some $i\in\{0,1,\cdots,\pi(q')-1\}$. Since $\mathcal{W}_{\delta}^{s}(q')=W^s_{\delta}(q')$, the point $y$ also has stable manifolds of $\delta$-size. Therefore, every point $x\in\mathcal{P}$ has stable manifolds of uniformly size.    
	    
	For any $b\in H(p)$, by the choice of the set $B$, there is $x\in B$ such that $d(b,x)<\frac{\varepsilon}{2}$. For this point $x$, there is a $q_x\in\mathcal{P}$ such that $d(x,q_x)<\frac{\varepsilon}{2}$. Then, $d(b,q_x)<\varepsilon$. Therefore, $\mathcal{P}$ is a $\varepsilon$-dense subset of $H(p)$.  
\end{proof}

\section{The closing property of thin trapped in $H(p)$}

  It is well-known that pseudoorbits near a hyperbolic set can be shadowed by a real orbit. This is called Pseudo-Orbit Tracing Property. This property plays an important role in the study of stability of dynamical systems (see \cite{Wen96}, \cite{Gan}, \cite{WL} and  \cite{Yu99} ). Gan \cite[Theorem 1.1]{GS} showed that quasi-hyperbolic pseudoorbits can be shadowed by a real orbit. In this section, we introduce the Pseudo-Orbit Tracing Property in $H(p)$. Before heading to the main block, we clarify some notations and identify some constants. 
     
\begin{Definition}\label{Def: thin}
  Assume that the homoclinic class $H(p)$ admits the dominated splitting $T_{H(p)} M=E\bigoplus_< F$, for $x\in H(p)$, $n\in\mathbb{N}^+$, an orbit arc $(x,n)\triangleq\{x,f(x),\cdots,f^{n-1}(x)\}$ is called $\lambda$-thin trapped, if $E$ is thin trapped and $$\prod\limits_{j=1}^{k}\|Df^{-1}|_{F_{f^{n-j}(x)}}\|\leq\lambda^k,~\forall~k=1,\cdots,n-1.$$
\end{Definition}

\begin{Definition}\label{Def: thin trapped}
   Assume that the homoclinic class $H(p)$ admits the dominated splitting $T_{H(p)} M=E\bigoplus_< F$, a finite number of orbit segment $\{(x_i,n_i)\}_{i=0}^{m}$ is called $(\lambda,\eta)$-thin trapped closed pseudoorbit, if $(x_i,n_i)$ is $\lambda$-thin trapped for $i=0,1,\cdots,m$ and $d(f^{n_i-1}(x_i),x_{i+1})\leq\eta$, for $i=0,1,\cdots,m-1$ and $d(f^{n_m-1}x_m,x_0)<\eta$.
\end{Definition}

\begin{Definition} 
   For $\delta>0$ and $N\in \mathbb{N}^+$, a sequence of points $\{x_1,x_2,\cdots,x_N\}$ is called $\delta$-shadowed by a periodic point $x$, if $N$ is the period of $x$ and $d(f^n(x),x_n)<\delta$, for $1\leq n\leq N$.
\end{Definition}

From the next proposition, we can see that the thin trapped homoclinic classes $H(p)$, which admit the dominated splitting $T_{H(p)}M=E\bigoplus_<F$ with dim$(E)=$ind$(p)$, has local product structures. 

\begin{Proposition}(\cite[Lemma 3.4]{CP})\label{TR}
 Given $\varepsilon>0$, there exists $\delta>0$ such that for any $x,y\in H(p)$ with $d(x,y)<\delta$, $\mathcal{W}^{cs}_\varepsilon(x)$ and $\mathcal{W}^{cu}_\varepsilon(y)$ are transversally intersect at a single point belonging to $H(p)$, where $\mathcal{W}^{\ast}_\varepsilon(x)\subset\mathcal{W}^{\ast}(x)$ is centered at $x$ with length $2\varepsilon$, $\ast=cs{\rm \ or \ }cu$.
\end{Proposition}

Hereafter, we assume that the homoclinic classes $H(p)$ admits the dominated splitting $T_{H(p)}M=E\bigoplus_<F$ with dim$(E)=$ind$(p)$. For this dominated splitting, $E$ is thin trapped and $f$ is uniformly $F$-expanding at the period on all periodic points homoclinically related to $p$. One  can find dense hyperbolic periodic points which has stable and unstable manifolds of uniformly size.

\begin{Lemma}\label{ST}
  Given $\delta,~\varepsilon>0$, one can find a $\delta$-dense set $\mathcal{P}\subseteq H(p)$ of hyperbolic periodic points homoclinically related to the orbit of $p$, such that for any $x\in \mathcal{P}$, $\mathcal{W}^{cs}_\varepsilon(x)$ and $\mathcal{W}^{cu}_\varepsilon(x)$ are the stable and unstable manifolds of $x$, respectively.  
\end{Lemma}

\begin{proof}
  Under our assumptions, one can get the conclusions from the Proposition \ref{USM} and Proposition \ref{SM}. 	
\end{proof}

  From the Theorem \ref{SM}, one can see that $\mathcal{P}$ is a hyperbolic set. By the Definition \ref{Def:hyperbolic} and Definition \ref{Def:expanding}, for the dominated splitting $T_{H(p)}M=E\bigoplus_<F$, there are $C>0$, $\lambda\in(0,1)$ such that for any $x\in\mathcal{P}$, one has that $$\parallel Df^n\arrowvert_{E_x}\parallel\leq C\lambda^n,~\forall~n\in\mathbb{N}^+~{\rm and}~~\prod^{\pi(x)}_{j=1}\parallel Df^{-1}\mid_{F_{f^j(x)}}\parallel\leq C\lambda^{\pi(x)}.$$  Hereafter, taking $n_0$ be the smallest positive integer such that $C\lambda^{n_0}\leq\dfrac{1}{2}$, we give the definition of $(n_0,\varepsilon)$-closed pseudoorbit for the constant $\varepsilon$. 

\begin{Definition}\label{closed pseudoorbit}
   Assume that the homoclinic class $H(p)$ admits the dominated splitting $T_{H(p)} M=E\bigoplus_< F$. Given $\lambda\in(0,1)$ and $\varepsilon>0$, a $(\lambda,\varepsilon)$-thin trapped closed pseudoorbit $\{(x_i,n_i)\}_{i=0}^{m}$ is called $(n_0,\varepsilon)$-closed pseudoorbit, if $n_i\geq n_0$, for $i=0,1,\cdots,m$.	
\end{Definition}

\begin{Theorem}\label{shadowing}
	Given $\varepsilon>0$, there is a constant $\delta_0>0$ such that for any $\delta\in(0,\delta_0)$, if a finite number of orbit segment $\{(x_i,n_i)\}_{i=0}^{m}\subset \mathcal{P}$ is a $(n_0,\delta)$-closed pseudoorbit, then the $(n_0,\delta)$-closed pseudoorbit can be $\varepsilon$-shadowed by a periodic point.
\end{Theorem}

\begin{proof}
  For the given $\varepsilon>0$, by the Proposition \ref{TR}, there exists $\delta_1$ such that for any $x,y\in H(p)$ with $d(x,y)<\delta_1$, one has that $$\mathcal{W}^{cs}_\varepsilon(x) \pitchfork\mathcal{W}^{cu}_\varepsilon(y)\neq\emptyset~~{\rm and }~~\mathcal{W}^{cu}_\varepsilon(x) \pitchfork\mathcal{W}^{cs}_\varepsilon(y)\neq\emptyset.$$ Applying the Proposition \ref{TR} again, for the constant $\alpha=\min\{\varepsilon,\delta_1\}$, one obtains a constant $\delta_2$ such that $$\mathcal{W}^{cs}_\alpha(x) \pitchfork\mathcal{W}^{cu}_\alpha(y)\neq\emptyset~{\rm and }~\mathcal{W}^{cu}_\alpha(x) \pitchfork\mathcal{W}^{cs}_\alpha(y)\neq\emptyset~~\text{for~any~$x,y\in H(p)$~with $d(x,y)<\delta_2$}.$$ Due to the Lemma \ref{ST}, for this $\delta_2$, there is a $\delta_2$-dense subset $\mathcal{P}\subseteq H(p)$ of hyperbolic periodic points homoclinically related to the orbit of $p$, such that for any $x\in \mathcal{P}$, $\mathcal{W}^{cs}_\alpha(x)$ are the stable manifolds of $x$, denote by $\mathcal{W}^s_\alpha(x)$ and $\mathcal{W}^{cu}_\alpha(x)$ are the unstable manifolds of $x$, denote by $\mathcal{W}^u_\alpha(x)$.

  Let $\delta_0=\dfrac{\delta_2-\alpha/2}{2}$, given $\delta\in(0,\delta_0)$, if a finite number of orbit segment $\{(x_i,n_i)\}_{i=0}^{m}\subset \mathcal{P}$ is a $(n_0,\delta)$-closed pseudoorbit, then we can construct a sequence of points which are the intersection points of some stable manifolds and $\mathcal{W}^{cu}$ plaques.
  
  Since $\{(x_i,n_i)\}_{i=0}^{m}\subset \mathcal{P}$ is a $(n_0,\delta)$-closed pseudoorbit, by the Proposition \ref{TR}, one has that $$z_1\in\mathcal{W}^u_\alpha(f^{n_1}x_1)\pitchfork\mathcal{W}^s_\alpha(x_2)\neq\emptyset.$$For the  hyperbolic set $\mathcal{P}$, by the Definition \ref{Def:hyperbolic} and Definition \ref{Def:expanding}, one has that $\parallel Df^n\arrowvert_{E_x}\parallel\leq C\lambda^n$, for any $n\in\mathbb{N}^+$. Therefore, one has that $$d(f^{n_2}(x_2),f^{n_2}(z_1))\leq C\lambda^{n_2} d(x_2,z_1)\leq C\lambda^{n_0} d(x_2,z_1)\leq\dfrac{d(x_2,z_1)}{2}\leq\dfrac{\alpha}{2}.$$ Consequently, one has that $d(f^{n_2}(z_1),x_3)\leq d(f^{n_2}(x_2),f^{n_2}(z_1))+d(f^{n_2}(x_2),x_3)\leq\delta_2$. By the Proposition \ref{TR}, one has that $$z_2\in\mathcal{W}^{cu}_\alpha(f^{n_2}z_1)\pitchfork\mathcal{W}^s_\alpha(x_3)\neq\emptyset.$$Similarly, $$z_i\in\mathcal{W}^{cu}_\alpha(f^{n_i}z_{i-1})\pitchfork\mathcal{W}^s_\alpha(x_{i+1})\neq\emptyset,~\text{for $i=3,4,\cdots,m-1$ } \text{and}$$
  $$z_m\in\mathcal{W}^{cu}_\alpha(f^{n_m}z_{m-1})\pitchfork\mathcal{W}^s_\alpha(x_1)\neq\emptyset.$$             

  Since the compactness of the set $H(p)$, the limits of $\lim\limits_{k\rightarrow +\infty}f^{k(\sum\limits^{m}_{i=1} n_i)}(z_m)$ exists and the limits point $z=\lim\limits_{k\rightarrow +\infty}f^{k(\sum\limits^{m}_{i=1} n_i)}(z_m)$ also belongs to the set $H(p)$. Due to $$f^{\sum\limits^{m}_{i=1} n_i}(z)=f^{\sum\limits^{m}_{i=1} n_i}(\lim\limits_{k\rightarrow +\infty}f^{k(\sum\limits^{m}_{i=1} n_i)}(z_m))=\lim\limits_{k\rightarrow +\infty}f^{(k+1)(\sum\limits^{m}_{i=1} n_i)}(z_m)=z~,$$ the point $z$ is a periodic point. From our construction of the sequence of points which are the intersection points of some stable manifolds and $\mathcal{W}^{cu}$ plaques and the properties of the $(n_0,\delta)$-closed pseudoorbit, one has that the $(n_0,\delta)$-closed pseudoorbit can be $\varepsilon$-shadowed by the periodic point $z$. 
\end{proof}

\section{Proof of the Main Theorem}   

We prove the Main Theorem by contradiction under the assumptions that
\begin{itemize}
	\item $p$ is a hyperbolic periodic point;
	
	\item $H(p)$ admits a dominated splitting $T_{H(p)}M=E\bigoplus_< F$ with dim$(E)=$ind$(p)$;
	
	\item $E$ is thin trapped and $f$ is uniformly $F$-expanding at the period on the set of periodic points homoclinically related to $p$;
	
	\item $F$ is not uniformly expanding on $H(p)$.
\end{itemize}

\paragraph{Building closed pseudoorbit.}
\begin{Lemma}\label{Ob}
   For every $r\in(0,1]$, there exists a $r$-obstraction point $b$ in $H(p)$.
\end{Lemma}
\begin{proof}
    (By contradiction.) If the lemma is wrong, then there is $r\in(0,1]$ such that there is no $r$-obstraction point in $H(p)$ by the Definition \ref{obstrction point}. Therefore, for any $x\in H(p)$, there exists $n=n(x)>0$ such that $\prod_{i=1}^n\|Df^{-1}|_{F(f^i(x))}\|<r^n$. Let $U(x)$ be the neighborhood of $x$ such that for any $y\in U(x)$, one has that $\prod_{i=1}^{n(x)}\|Df^{-1}|_{F(f^i(y))}\|<r^{n(x)}$. Since $H(p)$ is a compact set, there exist finite points $x_1,\cdots,x_m$ such that $H(p)\subset\bigcup_{i=1}^m U(x_i)$. Let 
   $$N=\max_{1\leq i\leq m}\{n(x_i)\},~C=\max_{x\in H(p)}\{\frac{\|Df^{-1}|_{F(f(x))}\|}{r},\cdots,\frac{\prod_{i=1}^{N}\|Df^{-1}|_{F(f^i(x))}\|}{r^{N}}\}.$$For any $x\in H(p)$ and $n\in\mathbb{N}^+$, by splitting every orbit segment $(x,f^nx)$ in segments of the form $(f^ix,f^{n(f^ix)}(f^ix))$, one has that $\prod_{i=1}^{n}\|Df^{-1}|_{F(f^i(x))}\|\leq Cr^n$. Hence $F$ is uniformly expanding. This contradicts our hypothesis that $F$ is not uniformly expanding on $H(p)$.
\end{proof}

  Now we construct a periodic closed pseudoorbit $P$ as follows. Hereafter, we fixed a sequence numbers $0<\lambda<r_4<r_3<r_2<r_1\leq 1$ and $\varepsilon>0$. According to the Lemma \ref{ST}, one can find a $\varepsilon/2$-dense set $\mathcal{P}\subseteq H(p)$ of hyperbolic periodic points homoclinically related to the orbit of $p$, such that every point belonging to $\mathcal{P}$ has stable and unstable manifolds of uniformly size. By the Lemma \ref{Ob}, there is a $r_1$-obstruction point $b_1\in H(p)$. Therefore, by the Lemma \ref{infty}, $HT_{1}(q_n,\mu_2)\rightarrow\infty$ as $n\rightarrow\infty$. This means that there is a sequence periodic points in $H(p)$ and the period of those periodic points tends to infinity. Consequentlty, we may assume that the period of every periodic point in the $\varepsilon/2$-dense set $\mathcal{P}\subseteq H(p)$ is large enough.  

\paragraph{Step 1.}
 By the Lemma \ref{Ob}, there is a $r_1$-obstruction point $b_1\in H(p)$. Therefore, one can find a sequence $\{q_n:q_n\thicksim p\}$ such that $\lim\limits_{n\rightarrow\infty}q_n=b_1$. For $\varepsilon/2>0,~r_2<\lambda_1<r_1\leq 1$, by the Lemma \ref{obstrction}, there exists $N_1=N_1(\varepsilon,\lambda_1)>0$ such that if $(x,f^{N_1}(x))$ is a $\lambda_1$-obstruction segment, then $d(x,\Lambda(\lambda_1))<\varepsilon/2$. Taking a $x_1\in \mathcal{P}$ with $\Gamma_2(x_1,\lambda_1)-\Gamma_1(x_1,\lambda_1)-1\geq N_1$, we get uniform $\lambda_1$-strings $(x_1,f^{HT_1(x_1,\lambda_1)}(x_1))$, $(f^{HT_1(x_1,\lambda_1)}(x_1),f^{\Gamma_1(x_1,\lambda_1)}(x_1))$ and $\lambda_1$-obstruction segment $(f^{\Gamma_1(x_1,\lambda_1)}(x_1),f^{\Gamma_2(x_1,\lambda_1)-1}(x_1))$. Then, there exists a $\lambda_1$-obstruction point $b_2\in H(p)$ such that $d(b_2,f^{\Gamma_1(x_1,\lambda_1)}(x_1))<\varepsilon/2$.

\paragraph{Step 2.} 
 For the $\lambda_1$-obstruction point $b_2\in H(p)$, we also use $\{q_n\}$ denote the sequence $\{q_n:q_n\thicksim p\}$ such that $\lim\limits_{n\rightarrow\infty}q_n=b_2$. For $\varepsilon/2>0,~r_2<\lambda_2<\lambda_1<r_1$, by the Lemma \ref{obstrction}, there exists $N_2=N_2(\lambda_2,\varepsilon)>0$ such that if $(x,f^{N_2}(x))$ is a $\lambda_2$-obstruction segment, then $d(x,\Lambda(\lambda_2))<\varepsilon/2$. By the Lemma \ref{infty}, $HT_{1}(q_n,\mu_2)\rightarrow\infty$ as $n\rightarrow\infty$. Let $B$ be the subset of $\mathcal{P}$ such that for every $q_n\in B$, one has that $$r_3^{HT_{1}(q_n,\lambda_2)-1}\cdot m\cdot m^{\Gamma_1(x_1,\lambda_1)-HT_{1}(x_1,\lambda_1)}\geq r_4^{HT_{1}(q_n,\lambda_2)+\Gamma_1(x_1,\lambda_1)-HT_{1}(x_1,\lambda_1)},$$ where $m\triangleq\inf_{x\in H(p)}\|Df^{-1}|_{F(x)}\|$. Taking $x_2\in B$ with $\Gamma_2(x_2,\mu_2)-\Gamma_1(x_2,\mu_2)-1\geq N_2$, we get uniform $\lambda_2$-strings $(x_2,f^{HT_1(x_2,\lambda_2)}(x_2))$, $(f^{HT_1(x_2,\lambda_2)}(x_2),f^{\Gamma_1(x_2,\lambda_2)}(x_2))$ and $\lambda_2$-obstruction segment $(f^{\Gamma_1(x_2,\lambda_2)}(x_2),f^{\Gamma_2(x_2,\lambda_2)-1}(x_2))$. Therefore, there exists a $\lambda_2$-obstruction point $b_3\in H(p)$ such that $d(b_3,f^{\Gamma_1(x_2,\lambda_2)}(x_2))<\varepsilon/2$. 
\paragraph{Step 3.} For $\varepsilon>0$, taking $\lambda_j$ with $r_4<r_3<r_2<\cdots<\lambda_j<\lambda_{j-1}<\cdots<\lambda_2<\lambda_1<r_1$, where $j=1,2,\cdots$, by repeating the Step 2, one can get a sequence point $\{x_n\}$ satisfying the following properties: 
\begin{itemize}
	\item $(x_j,f^{HT_1(x_j,\lambda_j)}(x_j))$, $(f^{HT_1(x_j,\lambda_j)}(x_j),f^{\Gamma_1(x_j,\lambda_j)}(x_j))$ are uniform $\lambda_j$-strings;
	
	\item $(f^{\Gamma_1(x_j,\lambda_2)}(x_j),f^{\Gamma_2(x_j,\lambda_j)-1}(x_j))$ are $\lambda_j$-obstruction segment; 
	
	\item $d(x_j,f^{\Gamma_1(x_{j-1},\lambda_{j-1})}(x_{j-1}))\leq \varepsilon$;
	
	\item $r_3^{HT_{1}(x_j,\lambda_j)-1}\cdot m\cdot m^{\Gamma_1(x_{j-1},\lambda_{j-1})-HT_{1}(x_{j-1},\lambda_{j-1})}\geq r_4^{HT_{1}(x_j,\lambda_j)+\Gamma_1(x_{j-1},\lambda_{j-1})-HT_{1}(x_{j-1},\lambda_{j-1})};$
\end{itemize}
   From the above step, we construct a pseudoorbit which is not closed. As $H(p)$ is compact set, there exist $m_0\in\mathbb{N}$ and $k>0$, such that $$d(f^{HT_1(x_{m_0},\lambda_{m_0})}(x_{m_0}),f^{HT_1(x_{m_0+k},\lambda_{m_0+k})}(x_{m_0+k}))<\varepsilon.$$ Let $K_j=HT_1(x_{m_0+j},\lambda_{m_0+j}),~j=1,2,\cdots,k$. Let $y_{m_0+j}=f^{HT_1(x_{m_0+j},\lambda_{m_0+j})}(x_{m_0+j}),j=0,\cdots,k-1$ and $L_j=\Gamma_1(x_{m_0+j},\lambda_{m_0+j})-HT_1(x_{m_0+j},\lambda_{m_0+j}), j=0,1,\cdots,k-1$. Therefore, we get the closed pseudoorbit $P$ which is the union of uniform $\lambda_1$-strings as $$(y_{m_0},f^{L_0}(y_{m_0})),(x_{m_0+1},f^{K_1}(x_{m_0+1})),(y_{m_0+1},f^{L_1}(y_{m_0+1})),(x_{m_0+2},f^{K_2}(x_{m_0+2})),$$  $$\cdots,(y_{m_0+k-1},f^{L_{k-1}}(y_{m_0+k-1})),(x_{m_0+k},f^{K_k}(x_{m_0+k})).$$ where $y_{m_0+j}=f^{K_j}(x_{m_0+j})$ and $d(f^{L_j}(y_{m_0+j}),x_{m_0+j+1})<\varepsilon,~j=0,1,\cdots,k-1$.

\paragraph{Estimation about periodic orbit.}
  From the construction of closed pseudoorbit, by the Theorem \ref{shadowing}, for $\delta>0$, there exist $\eta_0=\eta_0(\lambda,\delta)>0$, such that for any $\eta\in(0,\eta_0]$, there exists a periodic point $\delta$-shadows $(\lambda,\eta)$-thin trapped closed pseudoorbit. In fact, we also has the following properties as the Lemma \ref{uniform size} for the periodic point.   

\begin{Lemma}\label{uniform size}
   For the fixed $r_4<r_3<r_1$, there is a $\delta_0>0$ such that all those periodic points which $\delta_0$-shadows $(\lambda,\eta)$-thin trapped closed pseudoorbit have stable and unstable manifolds of uniformly size. 	
\end{Lemma}  
\begin{proof}
  For the fixed $r_4<r_3<r_1$, by the Theorem \ref{neighbor of contraction}, there is a constant $\delta_0$ such that for any $x\in B(q_n,\delta_0)$, one has that $$\prod_{i=1}^{\Gamma_1(q_n,\lambda_n)}\lVert Df^{-1}\arrowvert_{F(f^ix)}\rVert\geq r_4^{\Gamma_1(q_n,\lambda_n)}~\text{and}\prod_{i=1}^{\Gamma_1(q_n,\lambda_n)}\lVert Df^{-1}\arrowvert_{F(f^ix)}\rVert\leq r_1^{\Gamma_1(q_n,\lambda_n)}.$$ By the Theorem \ref{shadowing}, for $\delta_0>0$, there exist $\eta_0=\eta_0(\lambda,\delta_0)>0$, such that for any $\eta\in(0,\eta_0]$, there exists a periodic point $\delta_0$-shadows $(\lambda,\eta)$-thin trapped closed pseudoorbit. In the construction of closed pseudoorbit, taking $\varepsilon=\eta_0/4$, one can construct $(\lambda,\eta_0)$-thin trapped closed pseudoorbit. Therefore, there exists a periodic point $y$ $\delta_0$-shadows $(\lambda,\eta_0)$-thin trapped closed pseudoorbit. This means $y\in B(q_n,\delta_0)$ for any $q_n$ in $(\lambda,\eta_0)$-thin trapped closed pseudoorbit. Therefore,  $$\prod_{i=1}^{\pi(y)}\|Df^{-1}|_{F(f^{i}(y))}\|\geq r_4^{\pi(y)}~\text{and}~ \prod_{i=1}^{\pi(y)}\|Df^{-1}|_{F(f^{i}(y))}\|\leq r_1^{\pi(y)}.$$ Since $T_{H(p)}M=E\bigoplus_{<}F$ is a dominated splitting, by the Definition \ref{DS}, one can take an suitable norm such that $\|Df|_{E(x)}\|\cdot\|Df^{-1}|_{F(f(x))}\|\leq\lambda_0$, where $0<\lambda_0<\lambda<1$. Hence
  $$\prod_{i=0}^{\pi(y)-1}\|Df|_{E(f^i(y))}\|\leq\frac{\lambda_0^{\pi(y)}}{\prod\limits_{i=1}^{\pi(y)}\|Df^{-1}|_{F(f^i(y))}\|}\leq(\frac{\lambda_0}{\lambda})^{\pi(y)}.$$ Taking $\widetilde{\lambda}=\max\{\lambda_0/\lambda,~r_1\}$, one has that $$\prod_{i=0}^{\pi(y)-1}\|Df|_{E(f^i(y))}\|\leq \widetilde{\lambda}^{\pi(y)}~\text{and}~\prod_{i=1}^{\pi(y)}\|Df^{-1}|_{F(f^{i}(y))}\|\leq \widetilde{\lambda}^{\pi(y)}.$$
  This means that those periodic points have stable and unstable manifolds of uniformly size.  
\end{proof}
  According to the Theorem \ref{SM}, by the choice of the set $\mathcal{P}$ of hyperbolic periodic points homoclinically related to the orbit of $p$, the hyperbolic periodic points belonging to $\mathcal{P}$ have stable and unstable manifold of uniformly size. By the Lemma \ref{uniform size}, those periodic points which $\delta_0$-shadows $(\lambda,\eta)$-thin trapped closed pseudoorbit have stable and unstable manifolds of uniformly size. Therefore, there is a $\delta_1>0$ such that if periodic point $\delta_1$-shadows $(\lambda,\eta)$-thin trapped closed pseudoorbit, then periodic point is homoclinically related to the hyperbolic periodic points that given in the construction of the closed pseudoorbit. Therefore, for $\delta=\min\{\delta_0,\delta_1\}$, where $\delta_0$ is given in the Lemma \ref{uniform size}, by the Theorem \ref{shadowing}, the closed pseudoorbit $P$ can be $\delta$-shadowed by a periodic point $\widetilde{p}$. And one also has that $W^s_\varepsilon(\widetilde{p})\pitchfork W^u_\varepsilon(p)\neq\emptyset$ and $W^s_\varepsilon(p)\pitchfork W^u_\varepsilon(\widetilde{p})\neq\emptyset$. This means that $\widetilde{p}$ is homoclinically related to $p$. By the Lemma \ref{uniform size}, one has the estimation $\prod_{i=1}^{\pi(\widetilde{p})}\|Df^{-1}|_{F(f^i(\widetilde{p}))}\|\geq r_4^{\pi(\widetilde{p})}$. This contradicts that $f$ is uniformly $F$-expanding at the period on $H(p)$. Consequently, the assumption that $F$ is not uniformly expanding on $H(p)$, is invalid. Therefore, $F$ is uniformly expanding on $H(p)$. Here, we finish the proof of the Main Theorem.
\bigskip  
  
  Now, we give the proof of the Theorem \ref{wpp} under the Main Theorem.
  
\begin{proof}
  (The proof of Theorem \ref{wpp}.) In the assumption of the Theorem \ref{wpp}, for the dominated splitting $T_{H(p)}=E\oplus_{<}F$, we may assume that the splitting $F$ splits in $F=E^c\oplus E^u$. Therefore, the Main Theorem shows that $f$ is not uniformly $F$-expanding at the period on all periodic points homoclinically related to $p$. 
	 
  For every $\varepsilon>0$, taking $r<1$ such that $\log(r^{-1})<\varepsilon$. Since $f$ is not uniformly $F$-expanding at the period on all periodic points homoclinically related to $p$, there is a periodic point $q$ homoclinically related to $p$ such that $\prod^{\pi(q)}_{i=1}\lVert Df^{-1}|_{F(f^i(q))}\rVert\geq r^{\pi(q)}$. Since dim$E^c=1$, one has $$\lVert Df^{\pi(q)}|_{E^c(q)}\rVert^{-1}=\prod^{\pi(q)}_{i=1}\lVert Df^{-1}|_{E^c(f^i(q))}\rVert\geq \prod^{\pi(q)}_{i=1}\lVert Df^{-1}|_{F(f^i(q))}\rVert\geq r^{\pi(q)}.$$ Therefor,~ $\frac{1}{\pi(q)}\log(\lVert Df^{\pi(q)}|_{E^c(q)}\rVert)\leq\varepsilon$.   
\end{proof}

\paragraph{Acknowledgments.} We would like to thank Dawei Yang for his useful suggestions. We would like to thank Yuntao Zang for his useful conversations.

\vskip 5pt

\noindent Wanglou Wu

\noindent School of Mathematical Sciences

\noindent Soochow University, Suzhou, 215006, P.R. China

\noindent wuwanlou@163.com, wanlouwu1989@gmail.com

\vskip 5pt

\noindent Bo Li

\noindent School of Mathematical Sciences

\noindent Soochow University, Suzhou, 215006, P.R. China

\noindent libo15962221581@163.com


\end{document}